\documentclass[12pt]{article}
\usepackage[margin=1in]{geometry}
\usepackage[usenames]{color}
\usepackage[colorlinks=true,
linkcolor=webgreen,
filecolor=webbrown,
citecolor=webgreen]{hyperref}

\definecolor{webgreen}{rgb}{0,.5,0}
\definecolor{webbrown}{rgb}{.6,0,0}

\usepackage{color}

\usepackage{url}
\usepackage{graphicx}
\usepackage{subcaption}

\usepackage{authblk}

\usepackage[T1]{fontenc}

\usepackage{amsmath}
\usepackage{amssymb}
\usepackage{amsthm}
\usepackage{mathrsfs}

\theoremstyle{plain}
\newtheorem{theorem}{Theorem}
\newtheorem{corollary}[theorem]{Corollary}
\newtheorem{lemma}[theorem]{Lemma}

\theoremstyle{definition}

\newtheorem{problem}{Problem}

\theoremstyle{remark}

\title{Extending Dekking's construction of an infinite
  binary word avoiding abelian $4$-powers}
\author[1]{James Currie}
\author[2]{Lucas Mol}
\author[1]{Narad Rampersad}
\author[3]{Jeffrey Shallit}
\affil[1]{
Department of Mathematics and Statistics,
University of Winnipeg,
515 Portage Ave.,
Winnipeg, MB R3B 2E9,
Canada,
\url{j.currie@uwinnipeg.ca}, \url{n.rampersad@uwinnipeg.ca}.
}
\affil[2]{
Department of Mathematics and Statistics,
Thompson Rivers University,
 805 TRU Way,
Kamloops, BC V2C 0C8,
Canada,
\url{lmol@tru.ca}.
}
\affil[3]{
School of Computer Science,
University of Waterloo,
Waterloo, ON N2L 3G1, Canada,
\url{shallit@uwaterloo.ca}.
}

\begin{document}
\maketitle
\begin{abstract}
  We construct an infinite binary word with critical exponent $3$ that
  avoids abelian $4$-powers.  Our method gives an algorithm to
  determine if certain types of morphic sequences avoid additive powers.
  We also show that there are
  $\Omega(1.172^n)$ binary words of length $n$ that avoid abelian
  $4$-powers, which improves on previous estimates.
\end{abstract}

\section{Introduction}
Two words $x$ and $y$ are \emph{abelian equivalent}
if $x$ and $y$ are anagrams of each other.  An \emph{abelian square}
(resp.\ \emph{abelian cube}) is a word $xy$ (resp.\ $xyz$) where $x$
and $y$ (resp.\ $x$, $y$, and $z$) are abelian equivalent.
More generally, an \emph{abelian $k$-power} is a word $x_1x_2\cdots x_k$,
where the $x_i$ are all abelian equivalent.

Erd\H{o}s \cite{Erd61} asked if there was an infinite word over a
finite alphabet that avoided abelian squares.  Evdokimov \cite{Evd68}
constructed such a word over a $25$-letter alphabet; Ker\"anen
\cite{Ker92} improved this to a $4$-letter alphabet.  Dekking
\cite{Dek79} showed that there is an infinite ternary word that avoids
abelian cubes and an infinite binary word that avoids abelian
$4$-powers.  This paper presents some results that strengthen or
extend Dekking's result on the binary alphabet.

The first direction we explore in trying to strengthen Dekking's
construction is to consider the simultaneous avoidance of abelian
$4$-powers and ordinary (i.e., non-abelian) powers.  Ordinary powers
are defined as follows.  Let $w = w_1w_2\cdots w_n$ be a word of
length $n$ and \emph{period} $p$; i.e., $w_i = w_{i+p}$ for
$i = 0, \ldots, n-p$.  If $p$ is the smallest period of $w$, we say
that the \emph{exponent} of $w$ is $n/p$.  We also say that $w$ is an
\emph{$(n/p)$-power} of \emph{order $p$}.  Words of exponent $2$
(resp.~$3$) are called \emph{squares} (resp.~\emph{cubes}).  Words of
exponent $>2$ (resp.~$>3$) are called \emph{$2^+$-powers}
(resp.~\emph{$3^+$-powers}). If ${\bf x}$ is an infinite sequence, we
define the \emph{critical exponent} of ${\bf x}$ as
$$ \sup \{ e \in \mathbb{Q}: \text{ there is a factor of } {\bf x}
\text{ with exponent } e \}. $$ One can then ask: \emph{What is the
  least critical exponent among all infinite binary words that avoid
  abelian $4$-powers?}

We answer this question in this paper, but our proof involves a method
for proving the avoidance of a variation on abelian powers, namely,
\emph{additive powers}.  An \emph{additive square} is a word $xy$ over
an \emph{integer alphabet} where $x$ and $y$ have the \emph{same
  length} and the sum of the letters of $x$ is the same as the sum of
the letters of $y$.  \emph{Additive cubes} and {additive $k$-powers}
are defined analogously.  Over the alphabet $\{0,1\}$, a word is an
abelian $k$-power if and only if it is an additive $k$-power, but on
larger alphabets these two notions are no longer equivalent.  Three
important papers on the avoidance of additive powers are: the paper of
Cassaigne, Currie, Schaeffer, and Shallit~\cite{CCSS14}, the paper
of Rao and Rosenfeld~\cite{RR18}, and the paper of Lietard and
Rosenfeld~\cite{LR20}.

Our first main result is the following:

\begin{theorem}\label{main_theorem1}
  There exists an infinite binary word with critical exponent $3$ that
  avoids additive/abelian $4$-powers.
\end{theorem}

Another problem that goes beyond the existence of infinite binary
words that avoid abelian $4$-powers is to estimate the number of such
words of length $n$.  Currie \cite{Cur04} showed that there are
$\Omega(1.044^n)$ binary words of length $n$ that avoid abelian
$4$-powers.  We improve this to the following:

\begin{theorem}\label{main_theorem2}
  There are $\Omega(1.172^n)$ binary words of length $n$ that avoid
  abelian $4$-powers.
\end{theorem}

\section{The construction for Theorem~\ref{main_theorem1}}\label{sec_add4pow}
Let $f:\{0,1,2\}^*\to \{0,1,2\}^*$ be the morphism defined by
$$f(0)=001,\quad f(1)=012,\quad f(2)=212$$
and let $g:\{0,1,2\}^*\to \{0,1\}^*$ be the morphism defined by
\begin{align*}
  g(0)&=0001001110010001100011\\
  g(1)&=0001001110011101100011\\
  g(2)&=0111001110011101100011.
\end{align*}

We will prove that both $f^\omega(0)$ and $g(f^\omega(0))$ avoid
additive $4$-powers.  We therefore introduce some notation related to
additive powers.  Let $u$ be a word over an integer alphabet.
Then we write:
\begin{itemize}
\item $|u|$ for the length of $u$;
\item $S(u)$ for the sum of the letters of $u$; and,
\item $\sigma(u)$ for the vector $[|u|, S(u)]^T$.
\end{itemize}
Also, if $h:X^*\to Y^*$ is a morphism, we define
$$W_h = \max_{a \in X} |h(a)|.$$
In particular, we have $W_f=3$ and $W_g=22$.

\begin{lemma}\label{sum_matrix}
  Let $X$ and $Y$ be subsets of the integers and let $h:X^* \to Y^*$
  be a morphism.  Suppose that for every $x \in X$ we have $|h(x)|=a+bx$
  and $S(h(x)) = c + dx$ for some integers $a,b,c,d$.  Define the matrix
  $$M_h = \begin{bmatrix}a&b\\c&d\end{bmatrix}.$$ Then for any word $u
  \in X^*$ we have $\sigma(h(u)) = M_h\sigma(u)$.  Consequently, if $M_h$ is
  invertible, we have $\sigma(u) = M_h^{-1}\sigma(h(u))$.
\end{lemma}

\begin{proof}
  The result clearly holds when $u$ is the empty word, so suppose $u$
  is non-empty and write $u=u'x$, where $x$ is the last letter of $u$.
  If the claim holds for $u'$, then we have
  \begin{align*}
    M_h\sigma(u) &= M_h\begin{bmatrix}|u'x|\\S(u'x)\end{bmatrix}\\
            &= \begin{bmatrix}a|u'x|+bS(u'x)\\c|u'x|+dS(u'x)\end{bmatrix}\\
            &= \begin{bmatrix}a(|u'|+1)+b(S(u')+x)\\c(|u'|+1)+d(S(u')+x)\end{bmatrix}\\
            &= \begin{bmatrix}a|u'|+a+bS(u')+bx\\c|u'|+c+dS(u')+dx\end{bmatrix}\\
            &= \begin{bmatrix}|h(u')|+|h(x)|\\S(h(u'))+S(h(x))\end{bmatrix}\\
            &= \begin{bmatrix}|h(u)|\\S(h(u))\end{bmatrix}\\
            &= \sigma(h(u)).\qedhere
  \end{align*}
\end{proof}
    
Note that for $x \in \{0,1,2\}$, we have $|f(x)|=3+0x$ and
$S(f(x))=1+2x$.  Thus, we have
$$M_f = \begin{bmatrix}3&0\\1&2\end{bmatrix}.$$
Similarly, for $x \in \{0,1,2\}$, we have $|g(x)|=22+0x$ and
$S(g(x))=9+2x$, so we have
$$M_g = \begin{bmatrix}22&0\\9&2\end{bmatrix}.$$
Clearly, both $M_f$ and $M_g$ are invertible.

We adapt the \emph{template method} (see \cite{CR12} for an early
version of this method, or \cite{RR18} for a much more powerful
version of the method) to deal with additive powers.  A
\emph{template} (for additive $4$-powers) is an $8$-tuple
$$t=[a_0,a_1,a_2,a_3,a_4,d_0,d_1,d_2],$$ where $a_i \in \{0,1,2\}^*$,
$|a_i|\le 1$, and the $d_i$ are in $\mathbb{Z}^2$. A word $w$ is an
\emph{instance} of $t$ if we can write $$w =a_0x_0a_1x_1a_2x_2a_3x_3a_4,$$
such that for each $i$, $d_i=\sigma(x_{i+1})-\sigma(x_i)$.
Let $h \in \{f,g\}$.
For $A \in \{0,1,2\}^*$, $|A|\le 1$, an \emph{$h$-split} of $A$ is a
triple $[p,a,s]$, such that $h(A)=pas$, and $a \in \{0,1,2\}^*$,
$|a|\le 1$.

Let
$$T=[A_0,A_1,A_2,A_3,A_4,D_0,D_1,D_2] \text{ and }
t=[a_0,a_1,a_2,a_3,a_4,d_0,d_1,d_2]$$
be templates.  We say that $T$ is an \emph{$h$-parent} of $t$
if the $A_i$ have $h$-splits $[p_i,a_i,s_i]$ such that
$$d_i = M_hD_i + b_i,$$
where $b_i = \sigma(s_{i+1}p_{i+2}) - \sigma(s_ip_{i+1})$.
We call a template
$t'$ an \emph{$h$-ancestor} of $t$ if $t'$ is in the reflexive and
transitive closure of the $h$-parent relation on $t$.

\begin{lemma}\label{forward}
  Let $h \in \{f,g\}$ and let $U \in \{0,1,2\}^*$.
  Suppose template $T=[A_0,A_1,A_2,A_3,A_4,D_0,D_1,D_2]$ is an $h$-parent of template $t=[a_0,a_1,a_2,a_3,a_4,d_0,d_1,d_2]$.
  If $U$ contains an instance of $T$, then $h(U)$ contains an instance
  of $t$.
\end{lemma}

\begin{proof}
Suppose that $U$ contains an instance $V=A_0X_0A_1X_1\cdots X_3A_4$
of $T$. Since $T$ is an $h$-parent of $t$, the $A_i$ have $h$-splits $[p_i,a_i,s_i]$ such that $d_i=M_hD_i+b_i$, where $b_i=\sigma(s_{i+1}p_{i+2})-\sigma(s_ip_{i+1})$.  Thus $u=h(U)$ contains a factor
$v=a_0x_0a_1x_1a_2x_2a_3x_3a_4$, where $x_i=s_ih(X_i)p_{i+1}$.  We have
\begin{align*}
\sigma(x_{i+1})-\sigma(x_i)
      &= \sigma(s_{i+1}h(X_{i+1})p_{i+2})-\sigma(s_ih(X_i)p_{i+1})\\
  &= \sigma(h(X_{i+1})) + \sigma(s_{i+1}p_{i+2}) -
    \sigma(h(X_i)) - \sigma(s_ip_{i+1})\\
  &= M_h\sigma(X_{i+1}) + \sigma(s_{i+1}p_{i+2}) -
    M_h\sigma(X_i) - \sigma(s_ip_{i+1})\\
  &= M_h(\sigma(X_{i+1})-\sigma(X_i)) + \sigma(s_{i+1}p_{i+2}) -
    \sigma(s_ip_{i+1})\\
      &= M_hD_i + b_i\\
      &=d_i
\end{align*}
Hence, the word $v$ is an instance of $t$.
\end{proof}

If $t$ is a template, we write $d_i = \left[d_i^{(0)}, d_i^{(1)}\right]$ and
define  the quantities $$\Delta(t) = \max\left\{\left|d_0^{(0)}\right|,
\left|d_1^{(0)}\right|, \left|d_2^{(0)}\right|\right\}$$ and
$$B_h(t) = 6+4(W_h-2)+6\Delta(t).$$

\begin{lemma}\label{backward}
  Let $h \in \{f,g\}$ and let $U \in \{0,1,2\}^*$.
  If $h(U)$ contains an instance $v$ of template $t$ of length at least
  $B_h(t)$, then there is an $h$-parent $T$ of $t$ such that $U$ contains
  an instance $V$ of $T$ and $|V|<|v|$.
\end{lemma}

\begin{proof}
Suppose that $v = a_0x_0a_1x_1a_2x_2a_3x_3a_4$
is a factor of $u=h(U)$ and is an instance of
$$t=[a_0,a_1,a_2,a_3,a_4,d_0,d_1,d_2].$$
If $|x_i| \geq W_h-1$ for each $i$, then each $x_i$ can be written in the form $x_i = s_ih(X_i)p_{i+1}$
for some $X_i$,  where $s_i$ is a suffix of some $h(A_i)$ and
$p_{i+1}$ is a prefix of some $h(A_{i+1})$.  The analysis for Lemma~\ref{forward}
is thus reversible, and $U$ contains an instance of
$T=[A_0,A_1,A_2,A_3,A_4,D_0,D_1,D_2]$ where the $A_i$ have $h$-splits 
$[p_i,a_i,s_i]$ and
$$D_i=M_h^{-1}(d_i-b_i).$$

To complete the proof it suffices to show that if $|v| \geq B_h(t)$,
then for each $i$ we have $|x_i| \geq W_h-1$.

Suppose to the contrary that for some $i$ we have $|x_i| \leq W_h-2$.
Let $\{\xi_0 \leq \cdots \leq \xi_3\} = \{|x_0|, \ldots, |x_3|\}$;
i.e., the $\xi_i$ are the lengths of the $x_i$ arranged in
non-decreasing order.  Our hypothesis then is that $\xi_0 \leq W_h-2$.
By the definition of $\Delta(t)$, we have
$\xi_i \leq \xi_0 + i\Delta(t)$ for $i = 0,1,2,3$.  Hence,
  $$\sum_{i=0}^3 |x_i| = \sum_{i=0}^3 \xi_i \leq \sum_{i=0}^3(\xi_0+i\Delta(t)) = 4\xi_0 + \Delta(t)\sum_{i=0}^3 i = 4\xi_0 +
  6\Delta(t) \leq 4(W_h-2)+6\Delta(t).$$ Consequently, we have
  $$|v| \leq 5 + \sum_{i=0}^3 |x_i| \leq 5+4(W_h-2)+6\Delta(t) < B_h(t),$$
which is a contradiction.
\end{proof}

\begin{theorem}\label{add4pow}
  The infinite word $g(f^\omega(0))$ contains no additive $4$-powers.
\end{theorem}

\begin{proof}
An additive $4$-power $v$ in $g(f^\omega(0))$ is an instance of the
template
$$t_0 = [\epsilon, \epsilon, \epsilon, \epsilon, \epsilon, [0,0],
[0,0], [0,0]].$$ Applying Lemma~\ref{backward} with $h=g$ and $t=t_0$
(and thus $\Delta(t)=0$ and $B_g(t)=6+4(22-2)+6(0)=86$), we see that
if $|v| \geq 86$, then the infinite word $f^\omega(0)$ contains an instance
$V$ of some $g$-parent $T$ of $t_0$. First, we verify by a brute force computation that $g(f^\omega(0))$ contains no additive $4$-power $v$ of length less than $86$. Next, we use a computer to build a list
$Anc$ of all possible $g$-parents of $t_0$.  We find that there are
$17056$ such $g$-parents.  We then compute the set of all
$f$-ancestors of these $g$-parents.  For each $g$-parent, we compute
all of its $f$-parents and add any new templates found to the list
$Anc$.  We find $48$ new templates, so our list now contains $17104$
templates.  We again compute $f$-parents of these $48$ new templates,
but we find that this results in no new templates.  The process
therefore terminates at this step with $Anc$ containing $17104$
templates.  It follows that $v$ is a factor of $g(f(V))$ for some
instance $V$ of a template $T$ in $Anc$.

Furthermore, we find that for every template $t$ in $Anc$, we have
$\Delta(t) \leq 2$.  Hence, applying Lemma~\ref{backward} iteratively
(with $h=f$ and $B_f(t) \leq 6+4(3-2)+6(2)=22)$), we see that if
$f^\omega(0)$ contains an instance $V$ of $T$, then it contains one
where $|V| \leq 21$.  Furthermore, Lemma~\ref{forward} guarantees that
$g(f(V))$ contains an additive $4$-power.

To complete the proof, it suffices to check that for every factor $V$
of $f^\omega(0)$ of length at most $21$, the word $g(f(V))$ does not
contain an additive $4$-power.  A brute force computation shows that
$f^6(0)$ contains all length-$21$ factors of
$f^\omega(0)$.  We can then examine the prefix $g(f^7(0))$ of
$g(f^\omega(0))$ and verify that it contains no additive $4$-powers,
which establishes the claim.
\end{proof}

\begin{theorem}\label{crit_exp}
  The infinite word $g(f^\omega(0))$ contains cubes but no $3^+$-powers.
\end{theorem}

\begin{proof}
  This is proved with Walnut~\cite{Walnut}.  Here are the Walnut
  commands:
\begin{verbatim}
morphism g "0->0001001110010001100011 1->0001001110011101100011
    2->0111001110011101100011":
morphism f "0->001 1-> 012 2->212":
promote aut f:
image SA g aut:

eval containsthreeplus "?msd_3 Ei, n (n>=1) & At (t<=2*n)
    => SA[i+t]=SA[i+n+t]":
# checks to see if the word contains a 3+-power (answers FALSE)

eval containscubes "?msd_3 Ei, n (n>=1) & At (t<2*n) => SA[i+t]=SA[i+n+t]":
# checks to see if there are cubes (answers TRUE)
\end{verbatim}
\end{proof}

Theorem~\ref{main_theorem1} now follows from Theorems~\ref{add4pow}
and \ref{crit_exp}.  It is optimal in the sense that the longest
binary word avoiding both abelian $4$-powers and (ordinary) $3$-powers
is the following word of length $39$:
$$001101011011001001101100100110110101100,$$
and there is only one such word (up to binary complement).

We should also point out that, in principle, the method of Rao and
Rosenfeld \cite{RR18} would also work to prove Theorem~\ref{add4pow}.
We initially attempted to apply this method, but the computation
produced a large number of templates and we were not able to complete
the required computer checks.

\section{Generalizing to additive $k$-powers}
In this section we sketch how to generalize the method of the
previous section to additive $k$-powers.  We first need to generalize
the definition of template.  Given an integer alphabet $X$, an
\emph{additive $k$-template} is a $2k$-tuple
$$t=[a_0,\ldots,a_k,d_0,\ldots,d_{k-1}],$$ where $a_i \in X^*$,
$|a_i|\le 1$, and the $d_i$ are in $\mathbb{Z}^2$. A word $w$ is an
\emph{instance} of $t$ if we can write $$w =a_0x_0a_1x_1\cdots x_{k-1}a_k,$$
such that for each $i$, $d_i=\sigma(x_{i+1})-\sigma(x_i)$.

If $h$ is a morphism defined on an integer alphabet, we define
\emph{$h$-split}, \emph{$h$-parent}, and \emph{$h$-ancestor} as in
Section~\ref{sec_add4pow}.  We need Lemma~\ref{sum_matrix} and
additive $k$-power analogues of Lemmas~\ref{forward} and
\ref{backward}.  We state these analogues next; their proofs are
straightforward generalizations of those of Lemmas~\ref{forward} and
\ref{backward}.  Our method only applies to morphisms that satisfy the
conditions of Lemma~\ref{sum_matrix}, so let us call any such morphism
$h$ a \emph{linear morphism} and let us use $M_h$ to denote the matrix
of Lemma~\ref{sum_matrix}.

\begin{lemma}\label{kforward}
  Let $h$ be a linear morphism and let $U \in X^*$.  Suppose additive
  $k$-template $T$ is an $h$-parent of additive $k$-template $t$.
  If $U$ contains an instance of $T$, then $h(U)$ contains an instance
  of $t$.
\end{lemma}

If $h$ is a morphism and $t$ is an additive $k$-template, we write $d_i = \left[d_i^{(0)}, d_i^{(1)}\right]$ and
define  the quantities $$\Delta(t) = \max\left\{\left|d_i^{(0)}\right|
  : i=0,\ldots,k-2\right\}$$ and
$$B_h(t) = k+2+k(W_h-2)+\frac{(k-1)k}{2}\Delta(t).$$

\begin{lemma}\label{kbackward}
  Let $h$ be a linear morphism such that $|h(a)|\geq 2$ for all
  $a \in X$ and such that $M_h$ is invertible. Let $U \in X^*$.  If
  $h(U)$ contains an instance $v$ of additive $k$-template $t$ of
  length at least $B_h(t)$, then there is an $h$-parent $T$ of $t$
  such that $U$ contains an instance $V$ of $T$ and $|V|<|v|$.
\end{lemma}

\begin{theorem}\label{finite_anc}
  Let $h$ be a linear morphism such that all eigenvalues of $M_h$ are
  larger than $1$ in absolute value and let $t$ be an additive
  $k$-template.  Then the set of $h$-ancestors of $t$ is finite.
\end{theorem}

\begin{proof}
  If the eigenvalues of $M_h$ are larger than $1$ in absolute value, then $M_h$ is
  invertible and its eigenvalues are smaller than $1$ in absolute
  value.  Let $\lambda_1$ and $\lambda_2$ be the eigenvalues of $M_h^{-1}$
  and let $M_h^{-1} = P^{-1}JP$, where $J$ is the Jordan form of $M_h$.  Then either
  \[
    J = \begin{bmatrix}\lambda_1&0\\0&\lambda_2\end{bmatrix}
    \quad\text{ or }\quad
    J = \begin{bmatrix}\lambda_1&1\\0&\lambda_1\end{bmatrix}.
  \]
  Hence,
  \[
    M_h^{-n} = P^{-1}\begin{bmatrix}\lambda_1^n&0\\0&\lambda_2^n\end{bmatrix}P
    \quad\text{ or }\quad
    M_h^{-n} = P^{-1}\begin{bmatrix}\lambda_1^n&n\lambda_1^{n-1}\\0&\lambda_1^n\end{bmatrix}P,
  \]
  and so $\sum_{i \geq 0} M_h^{-i}$ converges.
  
  Let $$t=[a_0,\ldots,a_k,d_0,\ldots,d_{k-1}]$$ be an additive
  $k$-template and let $$T=[A_0,\ldots,A_k,D_0,\ldots,D_{k-1}]$$
  be an $h$-ancestor of $t$ obtained by applying the $h$-parent
  relation $\ell$ times.  If $\ell=1$, then, as in the proof of
  Lemma~\ref{backward}, we have $D_i = M_h^{-1}(d_i-b_i)$,
  where there are only finitely many choices for $b_i$.  For larger
  $\ell$, by iteration, we find that $D_i$ has the form 
  \begin{equation}\label{anc_vector}
    D_i = c_\ell M_h^{-\ell} + c_{\ell-1}M_h^{-(\ell-1)} + \cdots +
    c_1M_h^{-1} + c_0,
  \end{equation}
  where the $c_j$ are all taken from a finite set of vectors.  Since
  $\sum_{i \geq 0} M_h^{-i}$ converges and since
  $D_i \in \mathbb{Z}^2$, we see that there are only
  finitely many possible vectors $D_i$, and hence, only finitely many
  $h$-ancestors $T$ of $t$.
\end{proof}

\begin{corollary}
  Let $f:X^*\to X^*$ and $g:X^*\to Y^*$ be linear morphisms over
  integer alphabets $X$ and $Y$ such that $g(f^\omega(a))$ is an
  infinite word for some $a \in X$.  If
  \begin{itemize}
    \item $|f(a)| \geq 2$ and $|g(a)| \geq 2$ for all $a \in X$,
    \item $M_f$ and $M_g$ are invertible, and
    \item the eigenvalues of $M_f$ are larger than $1$ in absolute
      value,
  \end{itemize}
  then it is possible to decide if $g(f^\omega(a))$ avoids additive
  $k$-powers.
\end{corollary}

\begin{proof}
We follow the procedure from the proof of Theorem~\ref{add4pow}.
 An additive $k$-power $v$ in $g(f^\omega(0))$ is an instance of the
additive $k$-template
$$t_0 = [\epsilon, \ldots, \epsilon, [0,0], \ldots, [0,0]].$$
We first check by a brute force computation that $g(f^\omega(a))$
contains no additive $k$-power of length less than
$B_g(t_0) = k+2+k(W_g-2)$.  As in the proof of Theorem~\ref{add4pow},
we then need to compute the set of $g$-parents of $t_0$ and the set
$Anc$ of all $h$-ancestors of each of the $g$-parents of $t_0$.  By
Theorem~\ref{finite_anc}, the set $Anc$ is finite and effectively
computable.  By Lemmas~\ref{kforward} and \ref{kbackward}, it now
suffices to verify that for every $t \in Anc$, no factor of
$f^\omega(a)$ of length less than $B_f(t)$ is an instance of $t$.
This is a finite (computer) check.
\end{proof}

\section{The construction for Theorem~\ref{main_theorem2}}

A \emph{(multi-valued) substitution} on the alphabet $\Sigma$ is a function $\theta:\Sigma^*\rightarrow 2^{\Sigma^*}$ that satisfies
\[
\theta(uv)=\{UV\colon\ U\in\theta(u),V\in \theta(v)\}
\]
for all $u,v\in\Sigma^*$.  We extend $\theta$ to $2^{\Sigma^*}$ by defining
\[
\theta(S)=\bigcup_{u\in S}\theta(u)
\]
for all $S\subseteq \Sigma^*$.  This extension allows us to compose substitutions.

A substitution $\theta$ on $\Sigma$ is \emph{abelian $n$-power-free} if every word $V\in\theta(v)$ avoids abelian $n$-powers whenever $v\in\Sigma^*$ avoids abelian $n$-powers.  
Currie~\cite{Cur04} established that there are $\Omega(1.044^n)$ binary words of length $n$ that avoid abelian $4$-powers by exhibiting an abelian $4$-power free substitution on $\{0,1\}$.  We improve this to $\Omega(1.172^n)$ by composing several abelian $4$-power free substitutions on $\{0,1\}$.  

We begin with a technical lemma.  We say that a substitution $\theta$ on $\Sigma$ is \emph{letter-wise uniform} if for every $a\in\Sigma$, we have $|u|=|v|$ for all $u,v\in\theta(a)$.

\begin{lemma}\label{Lemma:GeneralGrowthRate}
Suppose that there exists an abelian $4$-power free word $\normalfont{\textbf{w}}\in\{0,1\}^\omega$ in which the frequency of the letter $0$ exists and is equal to $\alpha$.
Let $\theta:\{0,1\}^*\rightarrow 2^{\{0,1\}^*}$ be an abelian $4$-power free substitution which is letter-wise uniform.  For $a\in\{0,1\}$, let $\ell_a=|u|$ for some $u\in\theta(a)$, and let $m_a=|\theta(a)|$.  Then for any fixed $\epsilon>0$, there are $\Omega\left(\beta^n\right)$ abelian $4$-power free binary words of length $n$, where
\[
\beta=(m_0^{\alpha-\epsilon}m_1^{1-\alpha-\epsilon})^{\frac{1}{(\alpha+\epsilon)\ell_0+(1-\alpha+\epsilon)\ell_1}}. \qedhere
\]
\end{lemma}

\begin{proof}
Let $\textbf{W}\in \theta(\textbf{w})$, and consider the prefix $V$ of $\textbf{W}$ of length $n$.  Write $V=V'p$, where $V'\in \theta(v)$ for some prefix $v$ of $\textbf{w}$ and $|p|\leq \max_{a\in\{0,1\}}(\ell_a-1)$.  By the assumption that $\theta$ is abelian $4$-power free, every word in $\theta(v)$ avoids abelian $4$-powers.

Fix $\epsilon>0$.  For $n$ sufficiently large, we have $(\alpha-\epsilon)|v|\leq |v|_0\leq (\alpha+\epsilon)|v|$.  Therefore, we have
\[
n-|p|=\ell_0|v|_0+\ell_1|v|_1\leq \big((\alpha+\epsilon)\ell_0+(1-\alpha+\epsilon)\ell_1\big)|v|,
\]
which gives 
\[
|v|\geq \frac{n-|p|}{(\alpha+\epsilon)\ell_0+(1-\alpha+\epsilon)\ell_1}.
\]
Thus we have
\begin{align*}
|\theta(v)|&=m_0^{|v|_0}\cdot m_1^{|v|_1}\\
&\geq m_0^{(\alpha-\epsilon)|v|}m_1^{(1-\alpha-\epsilon)|v|}\\
&=\left(m_0^{\alpha-\epsilon}m_1^{1-\alpha-\epsilon}\right)^{|v|}\\
&\geq \left(m_0^{\alpha-\epsilon}m_1^{1-\alpha-\epsilon}\right)^{\frac{n-|p|}{(\alpha+\epsilon)\ell_0+(1-\alpha+\epsilon)\ell_1}}\\
&=c\left[(m_0^{\alpha-\epsilon}m_1^{1-\alpha-\epsilon})^{\frac{1}{(\alpha+\epsilon)\ell_0+(1-\alpha+\epsilon)\ell_1}}\right]^n,
\end{align*}
where $c=(m_0^{\alpha-\epsilon}m_1^{1-\alpha-\epsilon})^{\frac{-|p|}{(\alpha+\epsilon)\ell_0+(1-\alpha+\epsilon)\ell_1}}$.
\end{proof}

Let $\theta_0,\theta_1:\Sigma^*\rightarrow 2^{\Sigma^*}$ be defined as follows:
\begin{align*}
\theta_0(0)&=\{0001\} & \theta_1(0)&=\{011,101\}\\
\theta_0(1)&=\{011,101\} & \theta_1(1)&=\{0001\}
\end{align*}
Note that $\theta_1$ is obtained from $\theta_0$ by swapping the images of $0$ and $1$.
Currie~\cite{Cur04} proved that $\theta_0$ is abelian $4$-power free by generalizing a result of Dekking~\cite{Dek79}.  It follows that $\theta_1$ is abelian $4$-power free as well.

Let $x\in\{0,1\}^*$, and write $x=x_1x_2\cdots x_k$, where the $x_i$ are in $\{0,1\}$.  Then we define $\theta_x=\theta_{x_k}\circ \theta_{x_{k-1}}\circ\cdots\circ \theta_{x_1}$. Since $\theta_0$ and $\theta_1$ are letter-wise uniform and abelian $4$-power free, it follows that $\theta_x$ is letter-wise uniform and abelian $4$-power free.  Further, for any fixed word $x\in\{0,1\}^*$, it is straightforward to compute the constants $\ell_0$, $\ell_1$, $m_0$, and $m_1$ described in the statement of Lemma~\ref{Lemma:GeneralGrowthRate} for the substitution $\theta_x$.

\begin{proof}[Proof of Theorem~\ref{main_theorem2}]
Let $h:\{0,1\}^*\rightarrow \{0,1\}^*$ be the morphism defined by $h(0)=0001$ and $h(1)=011$.  It is well-known that $h^\omega(0)$ avoids abelian $4$-powers~\cite{Dek79}.  By a straightforward calculation (see~\cite[Theorem~8.4.7]{AllShall2003}), the frequency of the letter $0$ in $h^\omega(0)$ is $\alpha=\frac{\sqrt{5}-1}{2}$.

For each $k\in\{1,2,\ldots, 15\}$ and every $x\in\{0,1\}^k$, we used a computer to apply Lemma~\ref{Lemma:GeneralGrowthRate} with the word $\textbf{w}=h^\omega(0)$, the substitution $\theta=\theta_x$, and the number $\epsilon=10^{-5}$ to obtain a number $\beta_x$ such that the number of abelian $4$-power free binary words of length $n$ is $\Omega(\beta_x^n)$.  For each $k\in\{1,2,\ldots,10\}$, the maximum value of $\beta_x$ was maximized by a unique word $x\in\{0,1\}^k$.  This word $x$, and the value of $\beta_x$, is shown in Table~\ref{Table:MaxBeta}.  In particular, this establishes Theorem~\ref{main_theorem2}.
\end{proof}

\begin{table}
\begin{center}
\begin{tabular}{l l l}
$k$ & $x$ & $\beta_x$\\\hline
$1$ & $1$ & 1.13503537\\
$2$ & $01$ & 1.15986115\\
$3$ & $101$ & 1.16840698\\
$4$ & $1101$ & 1.17123737\\
$5$ & $11101$ & 1.17195987\\
$6$ & $111101$ & 1.17220553\\
$7$ & $1111101$ & 1.17226224\\
$8$ & $11111101$ & 1.17228469\\
$9$ & $011111101$ & 1.17228931\\
$10$ & $111111101$ & 1.17229090\\
$11$ & $01111111101$ & 1.17229161\\
$12$ & $001011111101$ & 1.17229185\\
$13$ & $0001011111101$ & 1.17229194\\
$14$ & $00001011111101$ & 1.17229198\\
$15$ & $000101011111101$ & 1.17229199
\end{tabular}
\end{center}
\caption{The unique word $x\in\{0,1\}^k$ that maximizes $\beta_x$.  (The value of $\beta_x$ is truncated after $8$ decimal places.)}
\label{Table:MaxBeta}
\end{table}

When we performed the computations reported in
Table~\ref{Table:MaxBeta}, we didn't just use $\theta_0$ and
$\theta_1$, but we also included substitutions $\theta_2$ defined by
$\theta_2(0)=\{0111\}$, $\theta_2(1)=\{001,010\}$, and $\theta_3$
defined by swapping the images of $\theta_2(0)$ and $\theta_2(1)$,
which both satisfy the criteria of Currie~\cite{Cur04} to be abelian
$4$-power free.  However, there was always an optimal composition that
only involved $\theta_0$ and $\theta_1$.

The current best known upper bound for the number of abelian $4$-power
free words over a binary alphabet is $O(1.374164^n)$ due to Samsonov and
Shur~\cite{SS12}.

\section{Open problems}
The set of all infinite binary words avoiding abelian $4$-powers is
still not completely understood and there are still problems to
investigate.  One is the following:

\begin{problem}
  What is the minimum possible frequency of
$0$'s in an infinite binary word avoiding abelian fourth powers?
\end{problem}

Some empirical calculations suggest it might be around $1/3$.  The
frequency of $0$'s in the binary complement of the word $h^\omega(0)$
from the proof of Theorem~\ref{main_theorem2} is
$(3-\sqrt{5})/2 \approx 0.381966$.


\begin{thebibliography}{99}
\bibitem{AllShall2003} J.-P. Allouche, J. Shallit, \textit{Automatic Sequences: Theory, Applications, Generalizations}, Cambridge, 2004.

\bibitem{CCSS14} J. Cassaigne, J. Currie, L. Schaeffer, J. Shallit,
  Avoiding three consecutive blocks of the same size and same sum,
  \textit{J. ACM} \textbf{61} (2014), 1--17.

\bibitem{Cur04}
J. Currie, The number of binary words avoiding abelian fourth powers
grows exponentially, \textit{Theoret. Comput. Sci.} \textbf{319}
(2004), 441--446.

\bibitem{CR12}
  J. Currie, N. Rampersad, Fixed points avoiding Abelian $k$-powers,
  \textit{J. Comb. Theory Ser. A} {\bf 119} (2012), 942--948.

\bibitem{Dek79} F. M. Dekking, Strongly non-repetitive sequences and
progression-free sets, {\it J. Comb. Theory Ser. A} {\bf 27}
(1979), 181--185.

\bibitem{Erd61}
P. Erd\H{o}s, Some unsolved problems, \emph{Magyar Tud. Akad.
Kutat\'o Int. K\"ozl.} \textbf{6} (1961), 221--254.

\bibitem{Evd68}
A. A. Evdokimov, Strongly asymmetric sequences generated by a finite number
of symbols, \emph{Soviet Math. Dokl.} \textbf{9} (1968), 536--539.

\bibitem{Ker92} V. Ker\"{a}nen, Abelian squares are avoidable on
4 letters, {\it Automata, Languages and Programming}, Lecture Notes
in Computer Science {\bf 623} (1992) Springer-Verlag, 41--52.

\bibitem{LR20}
  F. Lietard, M. Rosenfeld, Avoidability of additive cubes over
  alphabets of four numbers.  In \textit{Proc. DLT 2020}, Lecture Notes
  in Computer Science {\bf 12086} (2020) Springer-Verlag, 192--206.
  
\bibitem{Walnut}
H. Mousavi, Automatic theorem proving in Walnut.  Documentation (2016--2021) available at \url{https://arxiv.org/abs/1603.06017}~.

\bibitem{RR18}
M. Rao, M. Rosenfeld, Avoiding two consecutive blocks of same size and
same sum over $\mathbb{Z}^2$, \textit{SIAM J. Discrete Math.}
\textbf{32} (2018), 2381--2397.

\bibitem{SS12}
  A. Samsonov, A. Shur, On abelian repetition threshold, \textit{RAIRO
    Theor. Inf. Appl.} \textbf{46} (2012), 147--163.
\end{thebibliography}
\end{document}